\documentclass[11pt,leqno]{article}

\usepackage{amsmath,amsthm,amssymb}
\usepackage{graphicx,subfigure,booktabs,enumerate,url} 
\usepackage{mathabx}
\usepackage{algorithm,algpseudocode}
\usepackage{color}

\theoremstyle{plain}
\newtheorem{thm}{Theorem}[section]

\theoremstyle{remark}

\theoremstyle{definition}
\newtheorem{ex}[thm]{Example}

\newtheorem{assum}[thm]{Assumption}
\newtheorem*{prob1}{Problem $(\mathcal{C}_{\theta})$}
\newtheorem*{prob2}{Problem $(\mathcal{I})$}

\makeatletter

\numberwithin{equation}{section}
\numberwithin{figure}{section}
\numberwithin{table}{section}

\algnewcommand\algorithmicto{\textbf{to}}
\algdef{SE}[FOR]{ForTo}{EndFor}[2]{%
  \algorithmicfor\ #1\ \algorithmicto\ #2\ \algorithmicdo
}{%
  \algorithmicend\ \algorithmicfor
}

\title{\Large\bf Inverse stochastic optimal controls}
\author{Yumiharu Nakano\\[1em]
        \small{Department of Mathematical and Computing Science, School of Computing} \\
        \small{Tokyo Institute of Technology} \\
        \small{W8-28, 2-12-1, Ookayama, Meguro-ku, Tokyo 152-8550, Japan} \\
		\small{e-mail: nakano@c.titech.ac.jp}
}

\date{\today}

\begin{document}

\maketitle

\begin{abstract}
We study an inverse problem of the stochastic optimal control of general diffusions with performance index 
having the quadratic penalty term of the control process.  
Under mild conditions on the system dynamics, the cost functions, and the optimal control process, 
we show that our inverse problem is well-posed using a stochastic maximum principle. 
Then, with the well-posedness, we reduce the inverse problem to some root finding problem of the expectation of 
a random variable involved with the value function, which has a unique solution. 
Based on this result, we propose a numerical method for our inverse problem by replacing the expectation above with 
arithmetic mean of observed optimal control processes and the corresponding state processes. 
The recent progress of numerical analysis of Hamilton-Jacobi-Bellman equations enables the proposed method to be 
implementable for multi-dimensional cases. 
In particular, with the help of the kernel-based collocation method for Hamilton-Jacobi-Bellman equations, 
our method for the inverse problems still works well even when an explicit form of the value function is unavailable.   
Several numerical experiments show that the numerical method recovers the unknown 
penalty parameter with high accuracy. 
\begin{flushleft}
\end{flushleft}
\end{abstract}





\section{Introduction}\label{sec:1}

The {\it inverse optimal controls} refers to the problem of determining the performance index that makes a given control law optimal. 
Kalman \cite{kal:1964} first analyses the case of linear quadratic models in infinite horizon in details. 
Since then, a vast amount of research has been done on this subject. 
To name a few, Bellman and Kalaba \cite{bel-kal:1963} and Casti \cite{cas:1980} adopt the dynamic programming 
approach to derive some 
equations for cost functions using the feedback function of the optimal control for finite horizon problems. 
Thau \cite{tha:1967} and Anderson and Moore \cite{and-moo:1989} extend Kalman's results to 
the case of a nonlinear cost structure and of a multi input, respectively. 
In the literature of reinforcement learning, 
Ng and Russell \cite{ng-rus:2000} studies an inverse optimal control in the framework of 
Markov decision processes. See, e.g., 
Abbeel and Ng \cite{abb-ng:2004} and Ziebart et.al \cite{zie-etal:2008} for subsequent studies. 
In continuous-time stochastic optimal control, Dvijotham and Todorov \cite{dvi-tod:2010} works under a linearly solvable 
class of optimal control, where explicit forms of the optimal feedback laws are available, and 
Deng and Krsti{\'c} \cite{den-krs:1997b} and Do \cite{do:2019} adopt Lyapunov function approaches in infinite 
horizon problems. 
See the recent review paper Ab Azar et.al \cite{ab-etal:2020} for a comprehensive account.  

In this paper, we are concerned with an inverse problem of optimal stochastic controls of diffusion processes that 
are more general than those in \cite{dvi-tod:2010} and \cite{den-krs:1997b}. 
Our control problem is described as the minimization problem of 
\[
 \mathbb{E}\left[g(X(T)) + \int_0^T\left(f(t,X(t)) + \theta |u(t)|^2\right)dt\right] 
\]
over an admissible class of control processes $u$, where $\{X(t)\}_{0\le t\le T}$ denotes the controlled 
diffusion process corresponding to $u$. 
The performance of the state is measured by the functions $g$ and $f$, and 
$\theta>0$ is a parameter penalizing the magnitude of the control. See Section \ref{sec:2} below for a precise formulation.  
It should be emphasized that the performance criterion of this form appears in many applications, 
where it is trivial to determine the functions $g$ and $f$, but difficult to choose an appropriate value of $\theta$.  
For example, in trajectory control problems, we seek a state process that is as close as possible to a given curve, 
with a practical input process. 
Thus, in this case, $g$ and $f$ are given by distances between the state and the desired curve. 
Moreover, the penalty term $\theta |u(t)|^2$ for the control is naturally incorporated, which affects the overall performance 
depending on a choice of $\theta$. 
Motivated by this, we consider the problem of determining $\theta$ from given optimal controls for our minimization 
problem with unknown penalty parameter.    

Our first aim is, for the inverse problem above, to study the well-posedness defined by Hadamard \cite{had:1902}. 
It is widely accepted that a modern interpretation of Hadamard's principle is: the solution of a given problem exists, 
unique, and depends continuously on data. 
In our framework, the {\it existence} is interpreted as the surjectivity of the mapping from the set of unknown penalty 
parameters to the set of optimal controls, which is outside the scope of this paper since 
observed controls are assumed to be optimal for the problem above for some 
$\theta> 0$. 
The {\it uniqueness} is nothing but the bijectivity of the mapping. The continuous dependence on data, 
referred to as the {\it stability} in this paper, is the continuity of its inverse in some sense. 
See the statement of Theorem \ref{thm:2.5} below for a precise meaning of our stability.
As one might predict, 
the uniqueness and the stability do not hold in general in our continuous time optimal control problems with finite horizons. 
We will give such an example. 
On the other hand, under additional mild conditions on the drift, the diffusion, and the cost functions, 
and the optimal control process, 
we will show that the uniqueness and the stability do hold using the stochastic maximum principle. 
It should be emphasized that this is a first attempt to reveal sufficient conditions for which the inverse 
optimal control problems are well-posed for stochastic controls with finite horizons.

Our second aim is to propose a numerical method for the inverse problem above. The problem is, given $N$ independent samples 
$\{(u^{(j)}(t_i),X^{(j)}(t_i))\}_{i=0}^n$, $j=1,\ldots,N$, of the pair $\{(u^{*}(t_i),X^{*}(t_i))\}_{i=0}^n$ of the optimal control and 
the corresponding state processes, to determine $\theta$ computationally, where $0=t_0<t_1<\ldots < t_n=T$. 
Our approach is based on the simple optimality condition and is a reduction to 
some root-finding problem. 
This of course involves the value function computations, and so we rely on the recent progress of numerical analysis of 
Hamilton-Jacobi-Bellman equations. Thanks to the uniqueness result of the inverse problem, in our numerical experiments 
performed in Section \ref{sec:3}, the positive root $\theta$'s are recovered with high accuracy.     

The present paper is organized as follows: we introduce notation and formulate our inverse problem in Section \ref{sec:2}. 
Section \ref{sec:3} is devoted to the issue of the well-posedness. 
In Section \ref{sec:4}, we propose a numerical method for our inverse problems and validate it. Section \ref{sec:5} concludes. 

\section{Problem formulation}\label{sec:2}

First, we collect some notation used in this paper. 
Denote by $\mathrm{int}(U)$ the interior of $U$. 
Denote also by $D_{\xi}\varphi$ and $D_{\xi}^2\varphi$ the gradient vector and the Hessian matrix of $\varphi$ with respect to 
a variable $\xi$, respectively.  
We write $\mathbb{S}^d$ for the totality of symmetric $d\times d$ real matrices, and 
$a^{\mathsf{T}}$ for the transpose of a vector or matrix $a$. 
Let $T>0$ and $m\in\mathbb{N}$. 
Let $\{W(t)\}_{0\le t\le T}$ be an $m$-dimensional standard 
Brownian motion on a complete probability space $(\Omega,\mathcal{F},\mathbb{P})$. 
Further, let $X_0$ be a random variable that is independent of $\{W(t)\}_{0\le t\le T}$. 
Denote by $\{\mathcal{F}(t)\}_{0\le t\le T}$ the augmented filtration  
generated by $\{W(t)\}_{0\le t\le T}$ and $X_0$. 
The control processes are assumed to take values in a closed set 
$U$ of $\mathbb{R}^k$. 
The class $\mathcal{U}$ of controls is then defined by the set of all 
$U$-valued $\{\mathcal{F}(t)\}$-adapted processes $\{u(t)\}_{0\le t\le T}$ satisfying 
\[
 \mathbb{E}\int_0^T|u(t)|^2dt < \infty. 
\]

For any given $u\in\mathcal{U}$, consider the $d$-dimensional controlled stochastic differential equation 
\begin{equation}
\label{eq:1.1}
 dX(t) = b(t,X(t),u(t))dt + \sigma(t,X(t),u(t))dW(t) 
\end{equation}
with initial condition $X_0$, where the functions $b$ and $\sigma$ are assumed to be measurable such that 
the existence and uniqueness of \eqref{eq:1.1} is guaranteed. That is, we assume that for any $u\in\mathcal{U}$ there 
exists a unique continuous $\{\mathcal{F}(t)\}$-adapted process $\{X(t)\}_{0\le t\le T}$ satisfying 
\[
 X(t) =X_0 + \int_0^tb(s,X(s),u(s))ds + \int_0^t \sigma(s,X(s),u(s))dW(s), \quad 0\le t\le T, 
\]
almost surely. 
The functions $g:\mathbb{R}^d\to\mathbb{R}$ and $f:[0,T]\times\mathbb{R}^d\to\mathbb{R}$, the cost functions introduced in 
Section \ref{sec:1}, are both assumed to be measurable such that 
\[
 \mathbb{E}\left[g(X(T)) + \int_0^Tf(t,X(t))dt\right] > -\infty
\]
for any $u\in\mathcal{U}$. Then, 
we are concerned with the following optimal control problem: 
\begin{prob1}
Minimize  
\begin{equation*}
 J[u] := \mathbb{E}\left[g(X(T)) + \int_0^T\left\{f(t,X(t))+\theta|u(t)|^2\right\}dt\right] 
\end{equation*}
over $u\in\mathcal{U}$. 
\end{prob1}
Here we have assumed that the running cost is decomposed into that for the state and that for the penalty of 
the control input with weight parameter $\theta> 0$. 
Then, we consider the inverse problem with respect to $\theta$ given an optimal control and the corresponding state trajectory, 
i.e., our problem is to recover $\theta$ from a solution $\{u^*(t)\}\in\mathcal{U}$ for $(\mathcal{C}_{\theta})$ 
and the corresponding state 
process $\{X^*(t)\}$ under the assumption that $b$, $\sigma$, $g$ and $f$ are known. 

\section{Well-posedness}\label{sec:3}

We shall discuss Hadamard's well-posedness of the inverse problem of the 
stochastic optimal control. 
As stated in Section \ref{sec:1}, we skip the existence issue. 
Then, as in many other inverse problems, the uniqueness and the stability  
do not hold in general for continuous-time optimal controls. 
\begin{ex}
Consider the state equation 
\[
 \frac{dX(t)}{dt} = u^2(t)X(t), 
\]
with nonrandom initial condition $X_0$. 
The control processes $\{u(t)\}$ are taken from the class of $\mathbb{R}$-valued Borel functions on $[0,T]$. 
The objective function $J[u]$ is assumed to be given by 
\[
 J[u] = X(T) + \theta \int_0^Tu^2(t)dt, 
\]
where $\theta> 0$. Since $X(T)=X_0e^{\int_0^Tu^2(t)dt}$, we obtain 
$\inf_uJ[u]=\min_{\gamma\ge 0}\{X_0e^{\gamma}+\theta\gamma\}$. 
Thus, if $X_0\ge 0$ then $\gamma\equiv 0$, i.e., $u\equiv 0$, is optimal for any problem 
$(\mathcal{C}_{\theta})$ with $\theta> 0$. 
\end{ex}

To discuss the uniqueness, we impose the following: 
\begin{assum}
\label{assum:2.2}
\begin{enumerate}[\rm (i)]
\item The random variable $X_0$ is a constant. 
\item The set $U$ is compact and $\mathrm{int}(U)$ is nonempty. 
\item The functions $b$, $\sigma$, $f$, and $g$ are of $C^2$-class in $x$. 
\item There exists a constant $C_0>0$ and a modulus of continuity $\rho:[0,\infty)\to [0,\infty)$ such that 
for $\varphi(t,x,u) = b(t,x,u)$, $\sigma(t,x,u)$, $f(t,x)$, $g(x)$, we have   
\begin{align*}
 |\varphi(t,x,u)-\varphi(t,x^{\prime},u^{\prime})|
  &\le C_0|x-x^{\prime}| + \rho(|u-u^{\prime}|), \\ 
  |\varphi(t,0,u)|&\le C_0,  \\
 |D_x\varphi(t,x,u) - D_x\varphi(t,x^{\prime},u^{\prime})| 
  &\le C_0|x-x^{\prime}| + \rho(|u-u^{\prime}|), \\ 
 |D_x^2\varphi(t,x,u) - D_x^2\varphi(t,x^{\prime},u^{\prime})|
  &\le \rho(|x-x^{\prime}| + |u-u^{\prime}|),   
\end{align*}
for $t\in [0,T]$, $x,x^{\prime}\in\mathbb{R}^d$, and $u,u^{\prime}\in U$. 
\end{enumerate}
\end{assum}

It should be noted that Assumption \ref{assum:2.2} (i) means that the filtration $\{\mathcal{F}(t)\}_{0\le t\le T}$ is 
generated by the Brownian motion only. 
Assumption \ref{assum:2.2} (ii) excludes the case where $U$ is countable. 
It should be noted that whether this assumption holds true or not is easily confirmed in practice. 
The first two conditions in Assumption \ref{assum:2.2} (iv) are standard ones for stochastic optimal control problems. 
Under these two requirements, there exists a unique solution $\{X(t)\}_{0\le t\le T}$ of 
\eqref{eq:1.1} for each $u\in\mathcal{U}$. See Fleming and Soner \cite{fle-son:2006}. 

Under Assumption \ref{assum:2.2}, 
we can apply a stochastic maximum principle, as described in Yong and Zhou \cite{yon-zho:1999}, 
to obtain the uniqueness result for our inverse problem in the following sense:  
\begin{thm}
\label{thm:2.3}
Let Assumption $\ref{assum:2.2}$ hold. 
Suppose that $\{u^*(t)\}_{0\le t\le T}\in\mathcal{U}$ is optimal both for the problem 
$(\mathcal{C}_{\theta_1})$ and $(\mathcal{C}_{\theta_2})$ for some $\theta_1,\theta_2>0$. 
Moreover, suppose that 
$\mathbb{P}(u^*(t_0)\in \mathrm{int}(U)\setminus\{0\})>0$ for some $t_0\in [0,T]$. 
Then $\theta_1=\theta_2$.   
\end{thm}
\begin{proof}
Step (i). 
Since $(X^*(t),u^*(t))$ is optimal both for $(\mathcal{C}_{\theta_1})$ and $(\mathcal{C}_{\theta_2})$, 
for each $\theta=\theta_1$, $\theta_2$, there exists a unique solution $(p(t),q(t))$, $0\le t\le T$, of the backward stochastic 
differential equation (BSDE) 
\begin{equation}
\label{eq:2.1}
\begin{aligned}
 dp(t) &= -\bigg\{D_xb(t,X^*(t),u^*(t))^{\mathsf{T}}p(t) + \sum_{j=1}^m
  D_x\sigma_j(t,X^*(t),u^*(t))^{\mathsf{T}}q_j(t)  \\ 
  &\hspace*{2.5em} - D_xf(t,X^*(t))\bigg\}dt + q(t)dW(t), \\ 
 p(T) & = -D_xg(X^*(T)), 
\end{aligned}
\end{equation}
as well as there exists a unique solution $(P(t),Q(t))$, $0\le t\le T$, of the BSDE 
\begin{equation}
\label{eq:2.2}
\begin{aligned}
 dP(t) &= -\bigg\{D_xb(t,X^*(t),u^*(t))^{\mathsf{T}}P(t) + P(t)D_xb(t,X^*(t),u^*(t))^{\mathsf{T}} \\ 
  &\hspace*{2.5em} 
      + \sum_{j=1}^mD_x\sigma_j(t,X^*(t),u^*(t))^{\mathsf{T}}P(t)D_x\sigma_j(t,X^*(t),u^*(t))  \\ 
  &\hspace*{2.5em} 
      + \sum_{j=1}^m\Big\{D_x\sigma_j(t,X^*(t),u^*(t))^{\mathsf{T}}Q_j(t) 
      + Q_j(t)D_x\sigma_j(t,X^*(t),u^*(t))\Big\}  \\ 
  &\hspace*{2.5em} + D_x^2H_0(t,X^*(t),u^*(t),p(t),q(t))\bigg\}dt + \sum_{j=1}^mQ_j(t)dW(t), \\ 
 P(T) & = -D_x^2g(X^*(T)), 
\end{aligned} 
\end{equation} 
where $q(t)=(q_1(t),\ldots,q_m(t))$, $Q(t)=(Q_1(t),\ldots,Q_m(t))$, $\sigma_j(t,x,u)\in\mathbb{R}^d$ for each $j=1,\ldots,m$ 
such that $\sigma(t,x,u)=(\sigma_1(t,x,u),\ldots,\sigma_m(t,x,u))$, and 
\[
 H_0(t,x,u,p,q):= p^{\mathsf{T}}b(t,x,u)+\mathrm{tr}(q^{\mathsf{T}}\sigma(t,x,u)) 
  - f(t,x) -\theta|u|^2. 
\]
In particular, $p(t)$, $q_1(t),\ldots, q_m(t)$ are $\mathbb{R}^d$-valued and 
$Q_1(t),\ldots,Q_m(t)$ are $\mathbb{S}^{d}$-valued, all of which are adapted processes satisfying 
\[
 \mathbb{E}\int_0^T|\varphi(t)|^2dt < \infty 
\]
for $\varphi=p,q_1,\ldots,q_m, Q_1,\ldots,Q_m$. 
Moreover, with the generalized Hamiltonian 
\begin{align*}
\mathcal{H}(t,x,u) &:=H_0(t,x,u,p(t),q(t)) - \frac{1}{2}
 \mathrm{tr}\left[\sigma(t,X^*(t),u^*(t))^{\mathsf{T}}P(t)\sigma(t,X^*(t),u^*(t))\right] \\ 
 &\quad + \frac{1}{2}\mathrm{tr}\left\{[\sigma(t,x,u)-\sigma(t,X^*(t),u^*(t))]^{\mathsf{T}}P(t)
    [\sigma(t,x,u)-\sigma(t,X^*(t),u^*(t))]\right\}  
\end{align*}
we have 
\begin{equation}
\label{eq:2.3}
 \mathcal{H}(t,X^*(t),u^*(t))=\max_{u\in U}\mathcal{H}(t,X^*(t),u). 
\end{equation}
See Chapter 3 in \cite{yon-zho:1999}. 

Step (ii).  
We write $(p_i(t),q_i(t),P_i(t),Q_i(t))$ for the corresponding $(p(t),q(t),P(t),Q(t))$ for $\theta=\theta_i$ where $i=1,2$. 
Similarly, we write $\mathcal{H}_i$ for the corresponding $\mathcal{H}$ for $\theta=\theta_i$ where $i=1,2$.  
By the optimality condition \eqref{eq:2.3} and our assumptions, we obtain 
\[
 D_u\mathcal{H}_i(t_0,X^*(t_0),u^*(t_0)) = 0, \quad i=1,2, 
\]
with positive probability. So, 
\begin{align*}
 0 &= K(t_0,X^*(t_0),u^*(t_0),p_1(t),P_1(t)) - 2\theta_1u^*(t_0) \\ 
 &= K(t_0,X^*(t_0),u^*(t_0),p_2(t),P_2(t)) - 2\theta_2u^*(t_0)
\end{align*}
where for $t\in [0,T]$, $x, p\in\mathbb{R}^d$, and $P\in\mathbb{R}^{d\times d}$, 
\begin{align*}
 &K(t,x,u,p,P) \\ 
 &=D_u\Big[\frac{1}{2}\mathrm{tr}\left(\sigma(t,x,u)^{\mathsf{T}}P\sigma(t,x,u)\right) 
  + p^{\mathsf{T}}b(t,x,u)  \\ 
  &\quad + \frac{1}{2}\mathrm{tr}\left\{[\sigma(t,x,u)-\sigma(t,X^*(t),u^*(t))]^{\mathsf{T}}P(t)
    [\sigma(t,x,u)-\sigma(t,X^*(t),u^*(t))]\right\} \Big].  
\end{align*}
Since the uniqueness of the BSDEs immediately leads to $p_1=p_2$ and $P_1=P_2$, the equalities above yield 
$\theta_1=\theta_2$, as claimed. 
\end{proof}

To study the stability, we impose additional conditions. 
\begin{assum}
\label{assum:2.4}
The functions $b$ and $\sigma$ are of $C^1$-class in $u$. Further, 
for $\varphi(t,x,u) = b(t,x,u)$, $\sigma(t,x,u)$, $f(t,x)$, $g(x)$, we have   
\begin{align*}
 |\varphi(t,x,u)|+|D_u\varphi(t,x,u)|&\le C_0, \\ 
 |D_u\varphi(t,x,u) - D_u\varphi(t,x^{\prime},u^{\prime})|
  &\le \rho(|x-x^{\prime}| + |u-u^{\prime}|), \quad 
  x,x^{\prime}\in\mathbb{R}^d, \;\; u,u^{\prime}\in U, 
\end{align*}
where $C_0$ and $\rho$ are as in Assumption \ref{assum:2.2}. 
\end{assum}

Then, again by the stochastic maximum principle, we have the stability result in the following sense: 
\begin{thm}
\label{thm:2.5}
Let Assumptions $\ref{assum:2.2}$ and $\ref{assum:2.4}$ hold. 
Let $\{u^*(t)\}_{0\le t\le T}\in\mathcal{U}$ be optimal for the problem $(\mathcal{C}_{\theta^*})$ for some $\theta^*>0$, and 
for each $n\in\mathbb{N}$ let $\{u_n(t)\}_{0\le t\le T}\in\mathcal{U}$ be optimal for $(\mathcal{C}_{\theta_n})$ for some $\theta_n>0$. 
Suppose that 
\[
 \mathbb{E}\int_0^T|u_n(t)-u^*(t)|^2dt\to 0, \quad n\to\infty, 
\]  
and that there exist $n_0\in\mathbb{N}$ and a measurable set 
$E\subset [0,T]\times\Omega$ with positive measure satisfying 
\[
 u_n(t,\omega), \; u^*(t,\omega)\in\mathrm{int}(U)\setminus\{0\}, \quad (t,\omega)\in E, 
\]
for $n\ge n_0$. 
Then we have $\lim_{n\to\infty}\theta_n=\theta^*$. 
\end{thm}
\begin{proof}
By $C$ we denote a generic positive constant that may vary from line to line. 
For any $n\ge n_0$, by \eqref{eq:2.3}, 
\begin{equation}
\label{eq:2.4}
 2\theta_n u_n(t,\omega) - 2\theta^* u^*(t,\omega) = L_n(t,\omega) - L^*(t,\omega), \quad (t,\omega)\in E, 
\end{equation}
where $L_n(t,\omega)=K(t,X_n(t,\omega), u_n(t,\omega),p_n(t,\omega),P_n(t,\omega))$ with 
$X_n$ being the state process corresponding to $u_n$, and $(p_n,q_n)$ and $(P_n,Q_n)$ the solutions of the BSDE 
\eqref{eq:2.1} and \eqref{eq:2.2} with $u=u_n$ and $X=X_n$, respectively. 
$L^*$ is similarly defined. The usual arguments with Gronwall's lemma yield 
\begin{equation}
\label{eq:2.5}
 \sup_{0\le t\le T}\mathbb{E}|X_n(t)-X^*(t)|^2\le C\mathbb{E}\int_0^T\rho(|u_n(t)-u^*(t)|)^2dt. 
\end{equation}
Since $\rho$ is a modulus of continuity for $D_x^2b(t,\cdot,u)$ uniformly in $t$ and $u$, we may assume that 
$\rho$ is bounded. Hence, for any $\varepsilon>0$ there exists $\delta>0$ such that 
$\rho(r)\le \varepsilon + \delta r$, $r\ge 0$. From this and the $L^2$-convergence of $u_n$ we obtain 
\begin{equation}
\label{eq:2.6}
\sup_{0\le t\le T}\mathbb{E}|X_n(t)-X^*(t)|^2\to 0, \quad n\to\infty. 
\end{equation}
Furthermore, by the a priori estimates for the BSDE (see, e.g., \cite[Theorem 3.3, Chapter 7]{yon-zho:1999}), 
\begin{align*}
 &\mathbb{E}\int_0^T|p_n(t)-p^*(t)|^2dt  \\ 
 &\le C\mathbb{E}|D_xg(X_n(T))-D_xg(X^*(T))|^2 \\ 
 &\quad + C\mathbb{E}\int_0^T|D_xb(t,X_n(t),u_n(t))-D_xb(t,X^*(t),u^*(t))|^2|p^*(t)|^2dt  \\ 
 &\quad + C\mathbb{E}\int_0^T|D_x\sigma(t,X_n(t),u_n(t))-D_x\sigma(t,X^*(t),u^*(t))|^2|q^*(t)|^2dt  \\
 &\quad + C\mathbb{E}\int_0^T|D_xf(t,X_n(t))-D_xf(t,X^*(t))|^2 dt. 
\end{align*}
By Assumption \ref{assum:2.2} (iv) and \eqref{eq:2.5}, 
\begin{align*}
 &\mathbb{E}|D_xg(X_n(T))-D_xg(X^*(T))|^2 
 + \mathbb{E}\int_0^T|D_xf(t,X_n(t))-D_xf(t,X^*(t))|^2 dt \\ 
& \le C\mathbb{E}\int_0^T\rho(|u_n(t)-u^*(t)|)^2dt \to 0, \quad n\to\infty. 
\end{align*}
Using Assumption \ref{assum:2.2} (iv) again, for $\varepsilon>0$, we observe 
\begin{align*}
 &\mathbb{E}\int_0^T|D_xb(t,X_n(t),u_n(t))-D_xb(t,X^*(t),u^*(t))|^2|p^*(t)|^2dt \\ 
 &\le \mathbb{E}\int_0^T|D_xb(t,X_n(t),u_n(t))-D_xb(t,X^*(t),u^*(t))|^2|p^*(t)|^2 \\ 
 &\hspace*{3em} \times 1_{\{|X_n(t)-X^*(t)|+\rho(|u_n(t)-u^*(t)|)>\varepsilon\}}dt \\ 
 &\quad +  \mathbb{E}\int_0^T|D_xb(t,X_n(t),u_n(t))-D_xb(t,X^*(t),u^*(t))|^2|p^*(t)|^2 \\ 
 &\hspace*{3em} \times 1_{\{|X_n(t)-X^*(t)|+\rho(|u_n(t)-u^*(t)|)\le \varepsilon\}}dt \\ 
 &\le C\mathbb{E}\int_0^T|p^*(t)|^21_{\{|X_n(t)-X^*(t)|+\rho(|u_n(t)-u^*(t)|)>\varepsilon\}}dt 
     + C\varepsilon \mathbb{E}\int_0^T|p^*(t)|^2 dt.   
\end{align*} 
Thus, letting $n\to\infty$ and then $\varepsilon\to 0$, we have 
\[
 \lim_{n\to\infty}\mathbb{E}\int_0^T|D_xb(t,X_n(t),u_n(t))-D_xb(t,X^*(t),u^*(t))|^2|p^*(t)|^2dt=0. 
\]
Similarly, 
\[
 \lim_{n\to\infty}\mathbb{E}\int_0^T|D_x\sigma(t,X_n(t),u_n(t))-D_x\sigma(t,X^*(t),u^*(t))|^2|q^*(t)|^2dt=0. 
\]
Therefore, 
\begin{equation}
\label{eq:2.7}
 \lim_{n\to\infty}\mathbb{E}\int_0^T|p_n(t)-p^*(t)|^2 = 0. 
\end{equation}
By the same way, we obtain 
\begin{equation}
\label{eq:2.8}
 \lim_{n\to\infty}\mathbb{E}\int_0^T|P_n(t)-P^*(t)|^2 =0.  
\end{equation}
For notational simplicity we denote $D_ub_n(t)=D_ub(t,X_n(t),u_n(t))$. Analogously we use the notation 
$D_ub^*(t)$, $\sigma_n(t)$, and $D_u\sigma_n(t)$. With this notation, by Assumption \ref{assum:2.4} we see 
\begin{align*}
 &|L_n(t)-L^*(t)| \\ 
 &\le |D_ub_n(t)||p_n(t)-p^*(t)| + |p^*(t)||D_ub_n(t)-D_ub^*(t)| \\ 
 &\quad + C|\sigma_n(t)||P_n(t)||D_u\sigma_n(t)-D_u\sigma^*(t)| 
    + C|D_u\sigma^*(t)||P_n(t)||\sigma_n(t)-\sigma^*(t)| \\ 
 &\quad + |D_u\sigma^*(t)||\sigma^*(t)|)|P_n(t)-P^*(t)| \\ 
 &\le C(1+|X_n(t)| + |u_n(t)|)(|p_n(t)-p^*(t)|+|P_n(t)-P^*(t)|) \\ 
 &\quad + C_{\varepsilon}(1+|p^*(t)|+|p_n(t)|)(\varepsilon + |X_n(t)-X^*(t)| + |u_n(t)-u^*(t)|)
\end{align*}
for any $\varepsilon>0$ with constant $C_{\varepsilon}$ depending on $\varepsilon$. 
Thus, by Cauchy-Schwartz inequality and \eqref{eq:2.6}--\eqref{eq:2.8}, 
\[
 \lim_{n\to\infty}\mathbb{E}\int_0^T|L_n(t)-L^*(t)|dt = 0. 
\]
From this and \eqref{eq:2.4} it follows that  
\[
 2|\theta_n-\theta^*|\int_E|u_n|dt\times d\mathbb{P}\le 2|\theta^*|\mathbb{E}\int_0^T|u_n(t)-u^*(t)|dt + 
 \mathbb{E}\int_0^T|L_n(t)-L^*(t)|dt \to 0, \quad n\to\infty. 
\]
On the other hand, 
\[
 \lim_{n\to\infty}\int_E|u_n|dt\times d\mathbb{P} = \int_E|u^*|dt\times d\mathbb{P}>0. 
\]
Therefore, $\theta_n\to \theta^*$, as claimed.  
\end{proof}

Next we consider the case of the linear quadratic regulator problems. We need special treatment since 
Assumptions \ref{assum:2.2} and \ref{assum:2.4} (iii) exclude the case where the state processes are 
affine in controls taking values in unbounded sets.  

\begin{assum}
\label{assum:2.6}
\begin{enumerate}[\rm (i)]
\item The random variable $X_0$ satisfies $\mathbb{E}|X_0|^2<\infty$. 
\item The set $U$ is given by $U=\mathbb{R}^k$. 
\item The functions $b$, $\sigma$, $f$, and $g$ are given respectively by 
\begin{gather*}
 b(t,x,u) = b_0(t)x + b_1(t)u, \quad \sigma(t,x,u)=\sigma_0(t), \\ 
 f(t,x) = x^{\mathsf{T}}S(t)x, \quad g(x) = x^{\mathsf{T}}Rx,   
\end{gather*}
for $t\in [0,T]$, $x\in\mathbb{R}^d$, and $u\in\mathbb{R}^k$, where 
$b_0$ is $\mathbb{R}^{d\times d}$-valued, $b_1$ is $\mathbb{R}^{d\times k}$-valued, $\sigma_0$ is 
$\mathbb{R}^{d\times m}$, and $S$ is $\mathbb{R}^{d\times d}$-valued, all of which are continuous on $[0,T]$, and 
$R\in\mathbb{R}^{d\times d}$. 
Further $R$ and $S(t)$ are positive semidefinite for any $t\in [0,T]$.   
\end{enumerate}
\end{assum}

In the linear quadratic cases, explicit forms of the optimal controls are available. 
This helps us to obtain the following uniqueness:  
\begin{thm}
\label{thm:2.7}
Let Assumption $\ref{assum:2.6}$ hold. 
Suppose that $\{u^*(t)\}_{0\le t\le T}\in\mathcal{U}$ is optimal both for the problem 
$(\mathcal{C}_{\theta_1})$ and $(\mathcal{C}_{\theta_2})$ for some $\theta_1,\theta_2>0$. 
Moreover, suppose that 
$\mathbb{P}(u^*(t_0)\neq 0)>0$ for some $t_0\in [0,T]$. 
Then $\theta_1=\theta_2$.   
\end{thm}
\begin{proof}
For each $\theta=\theta_1$ and $\theta_2$, an optimal control for $(\mathcal{C}_{\theta})$ uniquely exists and we 
necessarily have 
\begin{equation}
\label{eq:2.9}
 u^*(t) = -\frac{1}{\theta}b_1(t)^{\mathsf{T}}F(t)X^*(t), \quad 0\le t\le T, 
\end{equation}
where $\{F(t)\}_{0\le t\le T}$ is a unique solution of the matrix Riccati equation 
\begin{equation}
\label{eq:2.10}
 \frac{d}{dt}F(t) + b_0^{\mathsf{T}}(t)F(t) + F(t)b_0(t) - \frac{1}{\theta}F(t)b_1(t)b_1(t)^{\mathsf{T}}F(t)+S(t) = 0, 
 \quad F(T)=R. 
\end{equation}
We refer to, e.g., Bensoussan \cite{ben:1992} for this result.  A simple application of It{\^o} formula then yields 
\[
 dF(t)X^*(t) = -(b_0(t)^{\mathsf{T}}F(t)+S(t))X^*(t)dt + F(t)\sigma_0(t)dW(t). 
\]

Let $F_1$ and $F_2$ be the solution of the Riccati equation \eqref{eq:2.10} corresponding to $\theta_1$ and $\theta_2$, 
respectively.  Then we have 
\[
 d(F_1(t)-F_2(t))X^*(t)=-b_0(t)^{\mathsf{T}}X^*(t)dt + (F_1(t)-F_2(t))\sigma_0(t)dW(t). 
\]
From $F_1(T)=F_2(T)$ it follows that 
\[
 (F_1(t)-F_2(t))X^*(t)=\mathbb{E}\left[\left.\int_t^Tb_0(s)^{\mathsf{T}}(F_1(s)-F_2(s))X^*(s)ds\right|\mathcal{F}(t)\right], 
\]
whence 
\[
 \mathbb{E}|(F_1(t)-F_2(t))X^*(t)|^2\le C\int_t^T\mathbb{E}|(F_1(s)-F_2(s))X^*(s)|^2ds 
\]
for some positive constant $C>0$. 
Therefore we have $F_1X^*=F_2X^*$ and so $\theta_1=\theta_2$ since 
$b_1(t_0)^{\mathsf{T}}F_1(t_0)X^*(t_0)\neq 0$ with positive probability. Thus the theorem follows. 
\end{proof}

As in the previous theorem, thanks to the explicit representation of the optimal controls, we have the 
following stability result:  
\begin{thm}
\label{thm:2.8}
Let Assumption $\ref{assum:2.6}$ hold. 
Let $\{u^*(t)\}_{0\le t\le T}\in\mathcal{U}$ be optimal for the problem $(\mathcal{C}_{\theta^*})$ for some $\theta^*>0$, and 
for each $n\in\mathbb{N}$ let $\{u_n(t)\}_{0\le t\le T}\in\mathcal{U}$ be optimal for $(\mathcal{C}_{\theta_n})$ for some $\theta_n>0$. 
Suppose that 
\[
 \mathbb{E}\int_0^T|u_n(t)-u^*(t)|^2dt\to 0, \quad n\to\infty, 
\]  
and that there exists a measurable set $E\subset [0,T]\times\Omega$ with positive measure satisfying 
\[
 u^*(t,\omega)\neq 0, \quad (t,\omega)\in E. 
\]
Then we have $\lim_{n\to\infty}\theta_n=\theta^*$. 
\end{thm}
\begin{proof}
By $C$ we denote a generic positive constant that may vary from line to line. 
Let $X^*(t)$ and $X_n(t)$ be the state processes corresponding to $u^*$ and $u_n$, respectively. 
Further, let $F^*$ and $F_n$ be the solution of the Riccati equation \eqref{eq:2.10} corresponding to $\theta^*$ and $\theta_n$, 
respectively, $n\in\mathbb{N}$. Then, as in the proof of Theorem \ref{thm:2.7}, 
\begin{align*}
 &F^*(t)X^*(t)-F_n(t)X_n(t) \\ 
 &=\mathbb{E}\left[\left.R(X^*(T)-X_n(T)) + \int_t^T b_0(s)^{\mathsf{T}}(F^*(s)X^*(s)-F_n(s)X_n(s))ds 
 \right|\mathcal{F}(t)\right], 
\end{align*}
whence by Gronwall inequality, 
\[
 \sup_{0\le t\le T}\mathbb{E}|F^*(t)X^*(t)-F_n(t)X_n(t)|^2\le C\mathbb{E}|X^*(T)-X_n(T)|^2\to 0, \quad n\to\infty. 
\]
Thus, 
\begin{align*}
 |\theta_n-\theta^*|\int_Eu_ndt\times d\mathbb{P}&\le |\theta^*|\mathbb{E}\int_0^T|u_n(t)-u^*(t)|dt \\ 
   &\quad + C\mathbb{E}\int_0^T|F^*(t)X^*(t)-F_n(t)X_n(t)|dt \to 0, \quad n\to\infty. 
\end{align*}
Consequently, $\theta_n\to\theta^*$, as required. 
\end{proof}

\section{Numerical method}\label{sec:4}

Here we propose a method for determining the weight parameter $\theta$ given observed 
data of optimal controls. To this end, recall that the value function $V$ for the problem 
$(\mathcal{C}_{\theta})$ is given by  
\begin{equation}
\label{eq:3.1}
 V(t,x;\theta) = \inf_{u\in\mathcal{U}_t}\mathbb{E}\left[\left. g(X(T)) 
 + \int_t^T(f(s,X(s)) + \theta |u(s)|^2)ds\right|X(t) = x\right],  
\end{equation}
for $(t,x)\in [0,T]\times\mathbb{R}^d$, 
where $\mathcal{U}_t$ is the set of all $U$-valued $\{\mathcal{F}(s)\}$-adapted processes 
$\{u(s)\}_{t\le s\le T}$ satisfying 
\[
 \mathbb{E}\int_t^T|u(s)|^2ds < \infty. 
\]
It is well-known that under Assumption \ref{assum:2.2} the value fuction
$V(t,x)\equiv V(t,x;\theta)$ is a unique continuous viscosity solution of 
the Hamilton-Jacobi-Bellman (HJB) equation 
\begin{equation}
\label{eq:3.2}
\left\{
\begin{split}
 &\partial_t V(t,x) + H(t,x,D_xV(t,x),D^2_xV(t,x))=0, \quad (t,x)\in [0,T)\times \mathbb{R}^d, \\
 &V(T,x)=g(x), \quad x\in\mathbb{R}^d, 
\end{split}
\right. 
\end{equation}
where 
\[
 H(t,x,p,M)=\inf_{u\in U}\left[b(t,x,u)^{\mathsf{T}}p 
 + \frac{1}{2}\mathrm{tr}(\sigma(t,x,u)\sigma(t,x,u)^{\mathsf{T}}M) 
 + f(t,x) + \theta|u|^2\right]
\]
for $(t,x,p,M)\in\mathbb{R}^d\times\mathbb{R}^d\times\mathbb{S}^d$ (see, e.g., \cite{fle-son:2006}).  
Under Assumption \ref{assum:2.6}, i.e., the case of the linear quadratic regulator problems, the value function 
$V$ is given by 
\begin{equation}
 V(t,x) = x^{\mathsf{T}}F(t)x + G(t), \quad (t,x)\in [0,T]\times\mathbb{R}^d, 
\end{equation}
where $F$ is as in \eqref{eq:2.10} and $G$ is the unique solution of 
\[
 \frac{d}{dt}G(t) + \mathrm{tr}(\sigma_0(t)^{\mathsf{T}}F(t)\sigma_0(t))=0, \quad G(T)=0. 
\] 

Then we have the following basic result: 
\begin{thm}
\label{thm:3.1}
Let $\{u^*(t)\}_{0\le t\le T}\in\mathcal{U}$ solves $(\mathcal{C}_{\theta^*})$ for some $\theta^*>0$, and  
$\{X^*(t)\}_{0\le t\le T}$ the corresponding state trajectory. 
Suppose that Assumption $\ref{assum:2.4}$ holds and $\mathbb{P}(u^*(t_0)\in \mathrm{int}(U)\setminus\{0\})>0$ for
 some $t_0\in [0,T]$.
Then, $\theta^*$ is a unique positive root of the equation
\begin{equation}
\label{eq:3.3}
 \Phi(\theta):=\mathbb{E}\left[g(X^*(T))+ \int_0^T(f(t,X^*(t)) + \theta |u^*(t)|^2)dt - V(0,X_0;\theta)\right]=0. 
\end{equation}
\end{thm}
\begin{proof}
Since $X_0$ is a constant and $\{u^*(t)\}$ is optimal, clearly we have 
$V(0,X_0;\theta^*)=\inf_{u\in\mathcal{U}}J[u]=J[u^*]$, whence $\Phi(\theta^*)=0$. 
If $\theta^{\prime}>0$ satisfies $\Phi(\theta^{\prime})=0$, then $\{u^*(t)\}$ is also optimal for $(\mathcal{C}_{\theta^{\prime}})$. 
By Theorem \ref{thm:2.3}, we obtain $\theta^{\prime}=\theta^*$. 
\end{proof}

We have the same result in the case of the linear quadratic regulator problems. 
\begin{thm}
\label{thm:3.2}
Let $\{u^*(t)\}_{0\le t\le T}\in\mathcal{U}$ solves $(\mathcal{C}_{\theta^*})$ for some $\theta^*>0$, and  
$\{X^*(t)\}_{0\le t\le T}$ the corresponding state trajectory. 
Suppose that Assumption $\ref{assum:2.6}$ holds and $\mathbb{P}(u^*(t_0)\neq 0)>0$ for
 some $t_0\in [0,T]$.
Then, $\theta^*$ is a unique positive root of the equation $\Phi(\theta)=0$, defined as in $\eqref{eq:3.3}$. 
\end{thm}
\begin{proof}
Since $\{u^*(t)\}$ is necessarily given by \eqref{eq:2.9}, it is clear that 
the process $\{u^*(t)|_{X_0=x}\}$ is optimal for any $x\in\mathbb{R}^d$. 
Moreover, the process $\{X^*(t)|_{X_0=x}\}$ is the state process corresponding to $\{u^*(t)|_{X_0=x}\}$. 
Thus, 
\begin{align*}
 J[u^*] &= \mathbb{E}\left[X^*(T)^{\mathsf{T}}RX^*(T)+\int_0^T\left(X^*(t)^{\mathsf{T}}P(t)X^*(t) 
  + \theta^*|u^*(t)|^2\right)dt\right] \\ 
  &= \mathbb{E}\left[\mathbb{E}\left[\left. X^*(T)^{\mathsf{T}}RX^*(T)+\int_0^T\left(X^*(t)^{\mathsf{T}}P(t)X^*(t) 
  + \theta^*|u^*(t)|^2\right)dt\right|X_0\right]\right] \\ 
  &=\mathbb{E}[V(0,X_0;\theta^*)]. 
\end{align*}
The rest of the proof can be done by the same way as in the proof of Theorem \ref{thm:3.1}. 
\end{proof}
 
Now, suppose that the $N$ independent samples 
$\{u^{(j)}(t_i), X^{(j)}(t_i)\}$, $i=0,\ldots,n$, $j=1,\ldots,N$, 
of optimal control process for $(\mathcal{C}_{\theta^*})$ at time $t_i$ and the corresponding 
state at time $t_i$ are available, where $0=t_0<\cdots < t_n=T$. 
An inverse control problem is then to determine the 
unknown $\theta^*$ from the observations. 
Our approach is to focus on the following problem: 
\begin{prob2}
Find a positive root of the equation 
\[
 \frac{1}{N}\sum_{j=1}^N\left\{g(X^{(j)}(T))+\sum_{i=0}^{n-1}(f(t_i,X^{(j)}(t_i))+\theta |u^{(j)}(t_i)|^2)
 (t_{i+1}-t_i) - V(0,X^{(j)}(0);\theta)\right\} =0. 
\]
\end{prob2}
In views of the strong law of large number and Theorems \ref{thm:3.1} and \ref{thm:3.2}, 
we can obtain an approximate solution of the inverse problem 
by solving the problem $(\mathcal{I})$ for sufficiently large $N$ and $n$.

\begin{ex}
\label{ex:3.3}
Consider the case where the state is described by the one-dimensional equation 
\[
 dX(t) = u(t)dt+\frac{1}{10} dW(t)
\]
with initial condition $X_0$ having the standard normal distribution, and 
the control objective $J[u]$ is given by 
\[
 J[u] = \mathbb{E}\int_0^1\left(10|X(t)|^2+\theta|u(t)|^2\right)dt, 
\] 
where $U=\mathbb{R}$.  
This problem is of course a particular case of the linear quadratic regulator ones. 
The value function $V$ is explicitly given by 
$V(t,x;\theta)=F(t;\theta)x^2 + G(t;\theta)$, where 
$F(t;\theta)=\sqrt{10\theta}\tanh ((1-t)\sqrt{10/\theta})$ and 
$G(t;\theta)=(\theta/100)\log(\cosh((1-t)\sqrt{10/\theta}))$. 
The unique optimal control $u^*(t)$ is given by 
\[
 u^*(t) = -\sqrt{\frac{10}{\theta}}\tanh\left(\sqrt{\frac{10}{\theta}}(1-t)\right)X^*(t). 
\]
To test our approach, we independently generate the samples $(X^{(j)}(t_i), u^{(j)}(t_i))$, 
$j=1,\ldots,10000$, $i=0,\ldots,1000$, of the optimal state and control for $\theta=1$, 
where $t_i=i/1000$, and consider these samples as observed data. 
We solve the root-finding problem by minimizing 
\[
 \frac{1}{10000}\left|\sum_{j=1}^{10000}\sum_{i=0}^{999}\left(10|X^{(j)}(t_i)|^2+ \theta|u^{(j)}(t_i)|^2\right)
 \Delta t - V(0,X_0^{(j)};\theta)\right|  
\]
over $\theta\in [0.0001,10]$. 
The estimated $\theta$ over $100$ trials has the mean $0.9971$ and the standard deviation 
$4.4181\times 10^{-4}$.    
\end{ex}

In most of nonlinear problems, analytical solutions of HJB equations are rarely available. 
So we need to numerically solve \eqref{eq:3.3} to approximate the value functions. 
Existing numerical methods applicable to (\ref{eq:3.3}) include the finite difference methods 
(see, e.g., Kushner and Dupuis \cite{kus-dup:2001} and 
Bonnans and Zidani \cite{bon-zid:2003}), the finite-element like methods 
(see, e.g., Camilli and Falcone \cite{cam-fal:1995} and 
Debrabant and Jakobsen \cite{deb-jak:2012}), 
the kernel-based collocation methods (see, e.g., Kansa \cite{kan:1990a,kan:1990b}, 
Huang et.al \cite{hua-etal:2006}, and Nakano \cite{nak:2017}), 
and the regression-based methods (see, e.g., Fahim et al.~\cite{fah-tou-war:2011},  E et al.~\cite{e-etal:2017}, 
Sirignano and Spiliopoulos \cite{sir-spi:2018}, and the references therein). 
It is well-known that the use of the finite difference methods is limited to low dimensional problems since 
the number of the spatial grids points has an exponential growth in dimension. 
In the kernel-based collocation methods, we seek an approximate solution 
of the form of a linear combination of a radial basis function 
(e.g., multiquadrics in the Kansa's original work) by solving finite dimensional linear equations. 
In general, this procedure allows for a simpler numerical implementation compared to the 
finite-element like methods and the regression-based methods. 
The regression-based methods, in particular the ones with neural networks, are prominent in high dimensional problems 
although they are computationally expensive. 

In Example \ref{ex:3.4} below, we will deal with a population control problem having a 3-dimensional state space. 
We adopt the kernel-based collocation methods to compute the value function for that problem. 

The kernel-based collocation methods are obtained by the interpolations with positive definite kernels  
applied backward recursively in time.   
Given a points set 
$\Gamma=\{x^{(1)},\ldots,x^{(N)}\}\subset\mathbb{R}^d$ such that 
$x^{(j)}$'s are pairwise distinct, and a positive definite function  
$\Phi:\mathbb{R}^d\to\mathbb{R}$, the function  
\[
 \sum_{j=1}^N(A^{-1}\psi|_{\Gamma})_j\Phi(x-x^{(j)}), \quad x\in\mathbb{R}^d, 
\]
interpolates $\psi$ on $\Gamma$. Here, $A=\{\Phi(x^{(j)}-x^{(\ell)})\}_{j,\ell=1,\ldots,N}$, 
$\psi|_{\Gamma}$ is the column vector composed of $\psi(x_j)$, 
$j=1,\ldots,N$, and 
$(z)_j$ denotes the $j$-th component of $z\in\mathbb{R}^N$.  
Thus, with time grid $\{t_0,\ldots,t_n\}$ such that $0=t_0<\cdots <t_n=T$, 
the function $\tilde{V}(t_n,\cdot)$ defined by 
\[
 \tilde{V}(t_n,x) = \sum_{j=1}^N(A^{-1}\tilde{V}_n)_j\Phi(x-x^{(j)}), \quad x\in\mathbb{R}^d
\]
approximates $g$, where $\tilde{V}_n=g|_{\Gamma}$. 
Then, for any $k=0,1,\ldots,n-1$, with a vector 
$\tilde{V}_{k+1}\in\mathbb{R}^N$ of an approximate solution at 
$\{t_{k+1}\}\times\Gamma$, we set  
\begin{align*}
 \tilde{V}(t_{k+1},x) &= \sum_{j=1}^N(A^{-1}\tilde{V}_{k+1})_j\Phi(x-x^{(j)}), \quad x\in\mathbb{R}^d, \\ 
 \tilde{V}_k &= \tilde{V}_{k+1} + (t_{k+1}-t_k)H_{k+1}(v_{k+1}^h), 
\end{align*}
where $H_{k+1}(\tilde{V}_{k+1}) = (H_{k+1,1}(\tilde{V}_{k+1}),\ldots,H_{k+1,N}(\tilde{V}_{k+1}))\in\mathbb{R}^N$ 
with
\[
 H_{k+1,j} = H(t_{k+1},x^{(j)}, D\tilde{V}(t_{k+1},x^{(j)}), D^2\tilde{V}(t_{k+1},x^{(j)})). 
\]
It should be mentioned that in view of the accuracy of interpolation, 
it is necessary to take $\Gamma$ densely. In particular, as the spatial dimension increases, 
$N$ must also become larger. This creates difficulties in numerically solving the linear equations 
$A\xi = \psi|_{\Gamma}$ or in computing the inverse matrix $A^{-1}$. 
On the other hand, basically there are no restrictions on the structure of $\Gamma$ in implementing the kernel-based 
collocation methods. 
As a result, they are useful in multi but relatively low dimensional problems. 
It should also be noted that the method above can be described in a matrix form. See Nakano \cite{nak:2020} for details.

\begin{ex}
\label{ex:3.4}
Consider the following simple SIR (Susceptible-Infectious-Recovery)  
epidemic model with vaccination studied in   
Rachah and Torres \cite{rac-tor:2015}: 
\begin{align*}
 \frac{dS(t)}{dt} &= -\beta S(t)I(t) - u(t) S(t), \\ 
 \frac{dI(t)}{dt} &=  \beta S(t)I(t) - \mu I(t), \\ 
 \frac{dR(t)}{dt} &= \mu I(t) + u(t) S(t),  
\end{align*}
with initial conditions $S(0)=0.95$, $I(0)=0.05$, $R(0)=0$, where $\beta=0.2$ and $\mu=0.1$. 
The control objective is given by 
\[
 J[u] = \int_0^{10}(10I(t) + \theta u^2(t))dt  
\]
with $U=[0,1]$. 
We assume that $\theta=1$ is a true parameter and solve the corresponding control problem 
by the kernel-based collocation method to obtain an approximate value function $\tilde{V}$,  nearly optimal trajectories 
$S^*(t_i)$, $I^*(t_i)$, $R^*(t_i)$, and a nearly optimal control $u^*(t_i)$, 
where $t_i=i/1000$, $i=0,1,\ldots,10000$. 
Specifically, the kernel-based method is performed with the 
following ingredients: the Gaussian kernel with $\alpha=1$ and 
the collocation points $\Gamma$ consisting of $8^3$ uniform grids points in 
$[S(0)-0.5, S(0)+0.5]\times [I(0)-0.5, I(0)+0.5]\times [R(0)-0.5,R(0)+0.5]$ including the boundary.   
With this method, we generate $100$ independent copies $X^{(j)}(t_i)$, $j=1,\ldots,100$, of 
$(S^*(t_i), I^*(t_i), R^*(t_i))+0.01\times Z_i$, 
$i=0,1,\ldots,n$, where $Z=(Z_0,\ldots,Z_n)$ is an $(n+1)\times 3$-dimensional Gaussian vector with all components being 
independent of each other and having zero mean and unit variance. 
Then we consider these samples with noises as observed data.  
As in Example \ref{ex:3.3}, we solve the root-finding problem $(\mathcal{I})$ by minimizing 
\[
 \frac{1}{100}\left|\sum_{j=1}^{100}\sum_{i=0}^{9999}\left(10|X^{(j)}(t_i)|^2+ \theta|u^{(j)}(t_i)|^2\right)
 \Delta t - \tilde{V}(0,X_0^{(j)};\theta)\right|  
\]
subject to $\theta\in [0.0001,10]$. 
The estimated $\theta$ over $100$ trials has the mean $0.9496$ and the standard deviation 
$0.0558$.     
\end{ex}

\section{Conclusion}\label{sec:5}

We study the inverse problem of the stochastic control of general diffusions with performance index 
\[
 \mathbb{E}\left[g(X(T)) + \int_0^T\left(f(t,X(t)) + \theta |u(t)|^2\right)dt\right]. 
\]
Precisely, we formulate the inverse problem as the one of determining the weight parameter $\theta>0$ 
given the dynamics, $g$, $f$, and an optimal control process. 
Under mild conditions on the drift, the volatility, $g$, and $f$, and under the assumption that the optimal control 
belongs to the interior of the control set, we show that our inverse problem is well-posed. 
It should be noted that whether the latter assumption holds true or not is easily confirmed in practice.  
Then, with the well-posedness, we reduce the inverse problem to some root finding problem of the expectation of 
a random variable involved with the value function, which has a unique solution. 
Based on this result, we propose a numerical method for our inverse problem by replacing the expectation above with 
arithmetic mean of observed optimal control processes and the corresponding state processes. 
Several numerical experiments show that the numerical method recovers the unknown $\theta$ with high accuracy. 
In particular, with the help of the kernel-based collocation method for Hamilton-Jacobi-Bellman equations, 
our method for the inverse problems still works well even when an explicit form of the value function is unavailable. 

Possible future studies include the well-posedness when the function $g$ or $f$ is also unknown, and 
numerical methods under non-uniqueness of the inverse problems. 
In the latter issue, we may need to use an additional criterion to determine unknowns from multiple candidates such as 
the maximum entropy principle adopted in \cite{zie-etal:2008}. 



\subsection*{Acknowledgements}

This study is supported by JSPS KAKENHI Grant Number JP17K05359.

\bibliographystyle{plain}
\bibliography{../mybib}

\end{document}